\documentclass[reqno, 11pt]{amsart}
\usepackage{tikz}
\usepackage{geometry}
\usepackage{graphicx}
\usepackage{amssymb,amsmath,amsthm}
\usepackage[hidelinks]{hyperref}
\usepackage{autonum}
\usepackage[toc]{appendix}

  { \end{itemize} }
  { \end{enumerate} }

\newcommand{\BS}[0]{\backslash}

\newcommand{\HW}[1]{} 

\renewcommand{\t}{\tilde}
\newcommand{\p}{\partial}
\renewcommand{\d}{\delta}

\newcommand{\SUS}{\subset}

\newcounter{pcounter}

\renewcommand{\AA}{{\mathcal A}}

\newcommand{\CC}{{\mathcal C}}

\newcommand{\EE}{{\mathcal E}}

\newcommand{\KK}{{\mathcal K}}

\newcommand{\OO}{{\mathcal O}}

\newcommand{\R}{\mathbb{R}}

\newcommand{\alp}{\alpha}

\newcommand{\eps}{\epsilon}

\newcommand{\lam}{\lambda}
\renewcommand{\phi}{\varphi}
\newcommand{\ome}{\omega}

\newcommand{\nin}{\not\in}

\DeclareMathOperator*{\dist}{{dist}}

\DeclareMathOperator*{\spt}{spt}

\DeclareMathOperator*{\SO}{SO}

\def\XXint#1#2#3{{\setbox0=\hbox{$#1{#2#3}{\int}$}
\vcenter{\hbox{$#2#3$}}\kern-.5\wd0}}

\newcommand{\lupref}[2]{\hspace{0ex} \stackrel{\eqref{#1}}{#2}} 

\usepackage{color}
\definecolor{verylightblue}{rgb}{0.95, 0.95, 0.95}  
\definecolor{lightblue}{rgb}{0.7, 0.7, 1}

\definecolor{eqyellow}{rgb}{0.9375,0.8984,0.5469}
\definecolor{subeqyellow}{rgb}{1,0.9373,0.8353}

\definecolor{mygreen}{rgb}{0.3, 0.6, 0.3} 
\definecolor{verylightgreen}{rgb}{0.95, 0.95, 0.95} 
\definecolor{verydarkgreen}{rgb}{0, 0.5, 0}
\definecolor{darkgreen}{rgb}{0.85, 0.85, 0.85}  
\definecolor{mydarkgreen}{rgb}{0, 0.5, 0} 

\definecolor{mybrown}{rgb}{0.85, 0.4, 0.3}
\definecolor{verylightbrown}{rgb}{0.98, 0.72, 0.58}
\definecolor{verydarkbrown}{rgb}{0.44, 0.26, 0.26}

\definecolor{orange}{rgb}{1, 0.5, 0}

\definecolor{BurntOrange}{rgb}{1,0.356,0}

\definecolor{mydarkred}{rgb}{1,0.086,0.255}

\definecolor{RoseVYDP}{rgb}{0.84,0.086,0.255}

\definecolor{dgreen}{rgb}{0, 0.8, 0.5}     

\definecolor{CanaryBRT}{rgb}{1,0.76,0.26}

\definecolor{cyan}{rgb}{0, 1, 1}
\definecolor{verylightgray}{rgb}{0.95, 0.95, 0.95}
\definecolor{verylightgray}{rgb}{0.95, 0.95, 0.95}
\definecolor{verylightred}{rgb}{1, 0.8, 0.78}
\definecolor{verylightyellow}{rgb}{0.99, 0.98, 0.5}

\newcommand{\mygreen}{\color{mygreen}}

\newcommand{\ignore}[1]{{}}

\newcommand{\NNN}[2]{\|#1\|_{#2}}

\newcommand{\NI}[1]{\|#1\|_{L^\infty}}
\newcommand{\NIL}[2]{\|#1\|_{L^\infty({#2})}}

\newcommand{\NPL}[3]{{\|#1\|_{L^{#2}(#3)}}}

\makeatletter
\newcommand{\VEC}[2][r]{
  \gdef\@VORNE{1}
  \left(\hskip-\arraycolsep%
    \begin{array}{#1}\vekSp@lten{#2}\end{array}%
  \hskip-\arraycolsep\right)}
\def\vekSp@lten#1{\xvekSp@lten#1;vekL@stLine;}
\def\vekL@stLine{vekL@stLine}
\def\xvekSp@lten#1;{\def\temp{#1}%
  \ifx\temp\vekL@stLine
  \else
    \ifnum\@VORNE=1\gdef\@VORNE{0}
    \else\@arraycr\fi%
    #1%
    \expandafter\xvekSp@lten
  \fi}
\makeatother

\usepackage{enumitem}

 \DeclareMathOperator{\Id}{Id}

 \newcommand{\al}{\alpha}
\newcommand{\DS}{\displaystyle} %

 \DeclareMathOperator{\tr}{tr}

\newtheorem{theorem}{Theorem}[section] \newtheorem{lemma}[theorem]{Lemma}

\newtheorem{proposition}[theorem]{Proposition}

\definecolor{darkcyan}{rgb}{0.0, 0.45, 0.95} 

\usepackage[normalem]{ulem}

\usepackage{marginfix} 
\newcounter{margcount} 
\newcounter{MODUS}
\ifnum\theMODUS=1
\newcommand{\DETAILS}[1]{{\mygreen{IN DETAIL:
      $\langle\hspace{-0.3ex}\langle$#1$\rangle\hspace{-0.3ex}\rangle$}}}
 \else
\newcommand{\DETAILS}[1]{} \fi

\title[Minimal Energy for Geometrically Nonlinear Elastic Inclusions]{Minimal Energy for Geometrically Nonlinear Elastic Inclusions in Two Dimensions}%

\author[I. Akramov]{Ibrokhimbek Akramov} \address{\textit{Ibrokhimbek Akramov:}
  Institute of Applied Mathematics, Heidelberg University, Im Neuenheimer Feld
  205, 69120 Heidelberg, Germany} \email{akramov@uni-heidelberg.de}

\author[H. Kn\"upfer]{Hans Kn\"upfer} \address{\textit{Hans Kn\"upfer:}
  Institute of Applied Mathematics, Heidelberg University, Im Neuenheimer Feld
  205, 69120 Heidelberg, Germany} \email{knuepfer@uni-heidelberg.de}

\author[M. {Kru\v{z}{\'\i}k}]{Martin {Kru\v{z}{\'\i}k}} \address{\textit{Martin {Kru\v{z}{\'\i}k}:}
  Institute of Information Theory and Automation, Czech Academy of Sciences, Pod
  vodárenskou ve\v{z}i 4, CZ-182 08, Prague 8, Czech Republic}
\email{kruzik@utia.cas.cz}

\author[A. R\"uland]{Angkana R\"uland} \address{\textit{Angkana R\"uland:}
  Institute of Applied Mathematics, Heidelberg University, Im Neuenheimer Feld
  205, 69120 Heidelberg, Germany} \email{Angkana.Rueland@uni-heidelberg.de}

\date{\today}

\begin{document}

\maketitle

\begin{abstract}
  We are concerned with a variant of the isoperimetric problem, which in our
  setting arises in a geometrically nonlinear two-well problem in
  elasticity. More precisely, we investigate the optimal scaling of the energy
  of an elastic inclusion of a fixed volume for which the energy is determined
  by a surface and an (anisotropic) elastic contribution. Following ideas from \cite{CS} and
  \cite{KnuepferKohn-2011}, we derive the lower scaling bound by invoking a
  two-well rigidity argument and a covering result. The upper bound follows from
  a well-known construction for a lens-shaped elastic inclusion.
\end{abstract}

\section{Introduction}

This article is concerned with a variant of the isoperimetric problem, for which
we investigate the optimal energy of an elastic inclusion of a fixed
volume. Here the energy consists of an interfacial and a \emph{geometrically
  nonlinear} elastic contribution. The latter is defined by an integral of the
stored-energy density function over a domain. As usual, the stored-energy
density depends on the strain and describes properties of the
material. Physically, the problem is motivated by nucleation phenomena which
arise, for instance, in shape-memory materials \cite{Bhattacharya-Book}.

\medskip

The set-up considered in this work is the geometrically nonlinear analogue of
the article \cite{KnuepferKohn-2011} where the isoperimetric problem for a
\emph{geometrically linear} elastic two-phase inclusion problem had been
investigated.  Our main aim is to deduce quantitative information on the
nucleation problem by studying its \emph{scaling properties}. The problem of
determining the sharp form of the inclusion seems to be more complicated
\cite{Mueller-Notes}. 
On top of the presence of non-quasiconvexity as in \cite{KnuepferKohn-2011}, in
the geometrically nonlinear setting under investigation, an additional
difficulty is present in the form of the \emph{nonlinear gauge group}
$\SO(2)$. In order to deal with this, we hence rely on the geometrically
nonlinear rigidity result from \cite{CS} in combination with the ideas from
\cite{KnuepferKohn-2011}.

\subsection{Model and statement of results} \label{sec-model} %

We consider the interior nucleation of a new phase in an elastic material in two
space dimensions. More specifically, we consider a material for which two different
phases (lattice structures) are energetically preferred. These are represented
by the  $\SO(2)$ orbit of the identity matrix $\Id \in \R^{2 \times 2}$  and the $\SO(2)$ orbit of another matrix $F \in \R^{2 \times 2}\backslash \SO(2)$.
 The deformation of the material is described by a function
$v : \R^2 \to \R^2$.   By the Cauchy--Born rule the energy of an elastic
material can be represented in terms of the gradient of the deformation function
$v $. 
Following the phenomenological theory of martensite and assuming Hooke's law, we study (volume-constrained) minimizers of the energy
\begin{align} \label{E} %
  \EE[\chi, v] \ = \ \int_{\R^2} |\nabla \chi|  + \int_{\R^2} (1 - \chi) {\rm dist}^2(\nabla v,\SO(2)) + \chi {\rm dist}^2(\nabla v,\SO(2)F) .
\end{align}
Here $\chi : \R^2 \rightarrow \{ 0 , 1 \}$ encodes the location of the new, minority phase. Its variation i.e. 
the first integral in \eqref{E} is the interfacial energy, while the second integral is
the elastic energy. Hence, our model includes penalizations of transitions between the phases and deviations 
from the corresponding material phase. 
We introduce $\mu>0$ to denote the volume of the inclusion
\begin{align} \label{volume} %
  \mu \ = \ \int_{\R^2} \chi \,
\end{align}
for the region $M:=\{x\in \R^2: \ \chi(x)=1\}$ associated with the minority
phase. In what follows, we will consider minimizers of the energy \eqref{E} for
a prescribed volume of the minority phase.  In order to rule out
self-intersections, as the set $\AA_m$ of admissible functions we consider
\begin{align}
  \AA_m \ := \ \big\{ (\chi, v) \in BV(\R^2, \{ 0, 1 \}) & \times H_{\rm loc}^1(\R^2,
  \R^2) \ : \  \\
  &v \text{ is bi-Lipschitz } \text{with constant } m \geq 1\big \}.
\end{align}

Here bi-Lipschitz with constant $m\geq 1$ means that $v$ is a homeomorphism and
$v$ and $v^{-1}$ are Lipschitz continuous functions with Lipschitz constants
$m\geq 1$.  Seeking to model nucleation phenomena, we assume that the strain $F$
is compatible with the identity matrix. In two dimensions this is equivalent to
the condition $\det F = 1$.  Our main result is the scaling of the minimal
energy for prescribed inclusion volume:

\begin{theorem}[Scaling of ground state energy] \label{thm-unscaled} %
Suppose that $F \in \R^{2 \times 2}\backslash\SO(2)$ satisfies $\det F =
  1$. Let $\mu$ be as in \eqref{volume}. Let $m \geq$ $\max$ $\{ \NNN{F}{}, \NNN{F^{-1}}{} \}$ $+$
  $C $ for some sufficiently large constant $C>0$. Then for any $\mu>0$ we have
  \begin{align} \label{E-scaling-unscaled} 
    \inf_{(\chi,v) \in \AA_m \text{satisfies } \eqref{volume}} \EE[\chi, v] \ %
    \sim \left\{ \begin{array}{ll}
                   \mu^{\frac{1}{2}} \qquad &\text{ for } \mu\leq 1,\\
                   \mu^{\frac{2}{3}} &\text{ for } \mu\geq 1.
    \end{array} \right.
  \end{align}
\end{theorem}
Here, we write $A \sim B$ means that $cA \leq B \leq CB$ for two constants
$c,C > 0$ which are independent of $\mu$ but may depend on $F$.  The first bound
corresponds to the usual isoperimetric regime in which the surface energy
dominates while the second estimate for $\mu\geq 1$ captures the effect of the
interaction of the surface and elastic energies. In particular, the role of
anisotropy in the elastic contribution in the form of the two physical phenomena
of \emph{compatibility} and \emph{self-accommodation} are captured in it. We do
not track the dependence on $\| F\|$ in the energy scaling behaviour.

\medskip

The result of Theorem \ref{thm-unscaled} confirms the similar scalings which
had been obtained in the framework of piecewise linear elasticity in
\cite{KnuepferKohn-2011}. In particular, the result shows that in the framework
of \emph{geometrically nonlinear} elasticity, the model imposes enough rigidity
to ensure the same lower bound on the energy as in the \emph{geometrically
  linear} model. This is in line with the fact that the only solution for the
two gradient problem for two compatible strains are twins in nonlinear
elasticity theory \cite[Prop. 2]{BallJames-1987} as well as in linear elasticity
theory. If we allow for more variants of martensite, the situation is expected
to become more intricate since in this case the corresponding many gradient
problems possibly allow for a large number of non-trivial solutions and
complicated microstructure \cite{Bhattacharya-Book, Mueller-Notes}.

\medskip 
\subsection{Ideas of the proof}

The proof of our main result can be split into two parts: an ansatz-free lower
bound estimate and an upper bound construction.  

On the one hand, in order to
verify the \emph{lower bound}, we observe that without loss of generality, we
may assume the deformation $F$ to be symmetric and positive-definite  after  using the polar decomposition theorem. 
By a suitable choice of coordinates, $F$ hence takes the form
\begin{align} \label{def-F} 
  F = \begin{pmatrix}
    \lam& 0\\
    0& \frac{1}{\lam}
  \end{pmatrix} \qquad 
       \text{for some }0<\lam < 1.
\end{align}
With these normalization results in hand, in the small volume regime, the lower
bound follows by the standard isoperimetric inequality. In the large volume
setting, we deduce the lower bound by a combination of a segment rigidity
argument from \cite{CS} (Section \ref{sec:rigidity}) and the localization argument from
\cite{KnuepferKohn-2011} (Subsection \ref{sec:proof_lower_bound}). 
Working with phase indicator energies as in \cite{KnuepferKohn-2011} or \cite{CapellaOtto-2010}, see \eqref{E}, contrary to the energies in \cite{CS}, we do not directly
control the full second derivatives of the deformation $v$. This additional degeneracy results in a
number of small adaptations becoming necessary. For settings with \emph{full} 
second derivative control the key localized energy estimate in Proposition 
\ref{prp-stray} would directly follow from Corollary 2.5 in \cite{CS}. Moreover, in this case also the higher dimensional problem could be treated directly in parallel by invoking the results from
\cite{ChermisiConti-2010} or \cite{Jerrard-Lorent13}. With our energies this would require adaptations of these strategies which we do not pursue in the present article.
The slightly stronger degeneracy of our energy which does not immediately yield the full second 
derivative control, also accounts for one of the
technical reasons for our bi-Lipschitz assumptions in the minimization problem;
another reason being the use of approximation theory for bi-Lipschitz functions
in Section \ref{sec:lower}.

On the other hand, the \emph{upper bound} is
derived by constructing a deformation $v$ corresponding to a well-known
construction for a lens-shaped elastic inclusion (see
e.g. \cite{KnuepferKohn-2011}) which in our geometrically nonlinear setting
leads to an orientation-preserving deformation.

\medskip

\subsection{Relation to the literature}

Due to their physical significance and the intrinsic mathematical interest in
``non-isotropic'' isoperimetric inequalities, nucleation problems for
shape-memory materials have been studied in various settings: In a
\emph{geometrically linearized} framework the compatible and incompatible
two-well problems (one variant of martensite and one variant of austenite) have
been considered in \cite{KnuepferKohn-2011}, where a localization strategy was
introduced. This also forms one of the two core ingredients of our
result. Moreover, the nucleation behaviour for the geometrically linearized
cubic-to-tetragonal phase transformation was studied in
\cite{KnuepferKohnOtto-2012} in which Fourier theoretic arguments in the spirit
of \cite{CapellaOtto-2009, CapellaOtto-2010} were exploited. Fourier theoretic arguments also underlie the study of the nucleation of multiple phases without gauge invariance in \cite{RT22}.
Using
related ideas, the nucleation behaviour at corners of martensite in an austenite matrix was investigated in \cite{BellaGoldman-preprint}. We also refer
to \cite{ball2016quasiconvexity, ball2013nucleation, kruvzik2013quasiconvexity} for the study of quasiconvexity at the
boundary. Further, highly symmetric, low energy nucleation mechanisms have been
explored in \cite{conti2017piecewise} and \cite{cesana2020exact} both in the
geometrically linear and nonlinear theories in two dimensions.  In the
\emph{geometrically nonlinear} settings substantially less is known in terms of
nucleation properties due to the presence of the nonlinear gauge group. In this
context, the incompatible two-well problem was studied in
\cite{chaudhuri2004rigidity} in which an incompatible two-well analogue of the
Friesecke--James--Müller rigidity result \cite{FJM-2002} was used. Moreover, the
study of model singular perturbation problems for the analysis of
austenite-martensite interfaces in terms of a surface energy parameter
\cite{KohnMueller-1992-1, KohnMueller-1992-2} laid the basis for an
intensive, closely related research on singular perturbation problems for
shape-memory alloys \cite{ChanConti-2015, Zwicknagl-2014, CS, Lorent2006,
conti2016low, conti2020energy, davoli2020two, chipot1995numerical, chipot1999sharp, ruland2016rigidity, ruland2022energy, ruland2021energy}. Contrary to the full
nucleation problems, in these settings the phenomenon of \emph{compatibility}
plays the main role, while nucleation phenomena in addition require the analysis
of the phenomenon of \emph{self-accommodation}.  Moreover, dynamic
nucleation results have been considered in \cite{kruvzik2005modelling,
della2021probabilistic, della2019analysis}. We refer to
\cite{Mueller-Notes} and \cite{kruvzik2019mathematical} for further references
on these and related results.

\medskip

Nonlocal isoperimetric inequalities have also been investigated for the
Ohta-Kawasaki energy and related models with Riesz interaction. We refer e.g. to
\cite{Knuepfer2012, Knuepfer2013, Lu2013, Julin2014, Bonacini2014, Bonacini2015,
  Frank2015, Julin2016,
  Alama2020}. 
In these models, above a critical volume minimizers do not exist anymore and the
scaling of the energy in terms of the mass is linear. Other related vectorial
models where the energy includes both interface type energies as well as a
(dipolar) nonlocal interaction are ferromagnetic systems. The nucleation of
magnetic domains during magnetization reversal and corresponding optimal
magnetization patterns have been investigated in
\cite{KnuepferMuratov-2011,KnuepferNolte-2018,KnuepferStantejsky-2022}, see also
\cite{OttoViehmann-2009}.  The competition between a nonlocal repulsive
potential and an attractive confining term is found also in other problems, for
example in models studying the interaction of dislocations
\cite{Meurs2021,Kimura2020} or \cite{Mora2018,Carrillo2019,Carrillo2020}.
Another anisotropic and nonlocal repulsive energy that has been treated
variationally using ansatz--free analysis is \cite{Carrillo2020} (based on
\cite{Mora2018,Carrillo2019}).  We finally briefly mention investigations of
other physical settings where related nonlocal isoperimetric inequalities have
been studied. This includes the works \cite{kohn2014optimal, kohn2016optimal,
  potthoff2021optimal} on compliance minimization, on epitaxial growth
(e.g. \cite{BellaGZ-2015}), on dislocations (e.g. \cite{ContiGO-2015}) and
superconductors (e.g. \cite{ChoksiCKO-2008,ContiGOS-2018}). We emphasize that
the above list of references is far from exhaustive.

\medskip

\subsection{Notation} We write $A \lesssim B$ if $A \leq C B$ for some 
constant $C$ which is independent of $\mu$, but may, for instance, depend on $F$. The Frobenius norm of a matrix $A\in\R^{d\times l}$ is denoted by
$\|A\|=\sqrt{\tr(A^tA)}$.  For two matrices $A,B$ we write
  $\dist(A,B) := \NNN{A-B}{}$, where $\NNN{\cdot}{}$ is the Frobenius norm, analogously, we define $\dist(A,\KK) := \dist_{K \in \KK}(A,K)$ for any
  $\KK \SUS \R^{2 \times 2}$.

\medskip

By $B_R(x)$ we denote the ball of radius $R>0$ centered at $x \in \R^2$ and we write $B_R := B_R(0)$. We write $M:=\spt\chi \subset \R^2$ to denote
the support of the minority phase.  For $E\subset \R^2$ and $v \in BV(E)$, the
total variation of $v$ is denoted by $\|\nabla v\|_E$.

\medskip

\section{Rigidity}
\label{sec:rigidity}
The aim of this section is to find a ``good'' set in the shape of a rhombus
which fulfills a variant of the rigidity estimate from \cite{CS}. We first
introduce some notation for the elastic energies for the deformation $v$. We set
\begin{align}
   e_{\rm elast}(\chi,v) \ : = \ (1 - \chi) {\rm dist}^2(\nabla v,\SO(2)) + \chi {\rm dist}^2(\nabla v,\SO(2)F).
\end{align}
Then the elastic energy for a 1D or 2D subset $E \SUS \R^2$ is defined as
\begin{align}\label{def-eps}
  \EE_{\rm elast}[\chi, v,E] \ 
  &:= \ %
    \int_{E} e_{\rm elast}(\chi,v) 
\end{align}
and the total elastic energy is
$\EE_{\rm elast}[\chi, v] := \EE_{\rm elast}[\chi, v,\R^2]$. Similarly, we
introduce
\begin{align}\label{def-eps_inv}
  \EE_{\rm elast}'[\chi, v,E] \ 
  &:= \ %
    \int_{E} (1 - \chi) {\rm dist}^2(\nabla v,\SO(2)) + \chi {\rm dist}^2(\nabla v,\SO(2)F^{-1}), 
\end{align}
which we will use in order to deal with estimates for the inverse of $v$.  If
the subset is one dimensional we integrate over the 1D Hausdorff measure instead of the Lebesgue measure.

\medskip

Before stating the central rigidity estimate, we formulate two auxiliary
lemmas. First, we note that there is a large set of non-singular points:
\begin{lemma}[Non-singular points] \label{lem-regpoint}%
  Let $f \in L^1(B_R)$ and $R>0$.  Then for any $\theta > 0$ there is
  $U \SUS B_R$ with $|B_R \BS U| < \theta$ and a constant $C = C(\theta)>0$ such
  that for any $x_0 \in U$ we have
\begin{align}
  \int_{B_R} |f(x)| \frac 1{\dist{(x,x_0)}} \ dx \ %
  \leq \ \frac{C}{ R} \ \NPL{f}{1}{B_R}.
    \end{align}
  \end{lemma}
  
  \begin{proof}
    This follows by an application of Fubini's theorem and since
    $\dist^{-1}{(\cdot,x_0)} \in L_{\rm loc}^1$.
  \end{proof}
By our bi-Lipschitz assumption, bounds on $v$ can be translated into analogous bounds for its inverse:
\begin{lemma}\label{aux-lem}
  Let $R>0$, $m\geq 1$ and let $(\chi, v)\in \AA_m$ with $v(0)=0$ and $v\in C^1(B_{mR})$. Assume that
\begin{align}
  \NPL{\chi}{1}{B_{mR}} \ \leq \ \eta R^2 \qquad %
  \text{and} \qquad %
 \NNN{\nabla \chi}{B_{mR}} \ \leq \ \eta R.
\end{align}
Then for $\chi_1 :=\chi \circ (v^{-1})$ we have
\begin{enumerate}
\item\label{invVB} $\DS \NPL{\chi_1}{1}{B_{R}} \ \leq \ m^2 \eta R^2;$
\item\label{invPB} $\DS \NNN{\nabla (\chi_1)}{B_{R}} \ \leq \ m \eta R;$
\item\label{invEB} %
  $\DS \EE'_{\rm elast}[\chi_1, v^{-1}, B_R]\ \leq \ C \EE_{\rm
    elast}[\chi, v, B_{mR}]$ for some constant $C=C(m, F)>0$.
\end{enumerate} 
\end{lemma}
\begin{proof}
By the transformation formula and since $v \in \AA_m$, \ref{invVB} follows from
\begin{equation}
\begin{split}
  \|\chi_1\|_{L^1(B_R)} \ 
  &\leq\int_{B_{mR}}\chi(y)|\det \nabla v(y)|\,dy \ %
  \leq \ m^2\NPL{\chi}{1}{B_{mR}}\leq m^2\eta R^2.
\end{split}
\end{equation} 
By the chain rule for BV functions (cf. Theorem 3.16 in
\cite{AmbrosioFuscoEtAl-2000}) this implies
\begin{align}
  \int_{B_{R}}|\nabla \chi_1| \ 
  \leq \ m   \int_{B_{mR}}|\nabla \chi| \ 
  = \ m\eta R.
\end{align}
The claim of \ref{invEB} follows by an
application of the linear algebra fact from Lemma \ref{lem-Fsym}
\ref{Fsym-lem2}. Indeed, using the pointwise identity
\begin{align}\label{eq:inverse}
  {\rm dist}^2((\nabla v)^{-1},\SO(2) A^{-1}) & \ \leq \ C {\rm dist}^2(\nabla                                          v,\SO(2) A) \qquad %
                                                \text{for $A \in \{ \Id, F \}$}    
\end{align} 
together with the inverse function theorem, the transformation theorem and with the notation $\t v := v^{-1}$, we arrive at
\begin{equation}
\begin{split}
  \hspace{6ex} & \hspace{-6ex} %
  \EE'_{\rm elast}[\chi_1, \t v, B_R] \
  = \int_{B_R} (1 - \chi_1(y)) {\rm dist}^2((\nabla v)_{|\t v(y)}^{-1},\SO(2)) + \chi_1(y) {\rm dist}^2((\nabla v)_{|\t v(y)}^{-1},\SO(2) F^{-1})\ dy \\
  &\stackrel{\eqref{eq:inverse}}{\leq \ } \ C \int_{B_{mR}}(1 - \chi(x)) {\rm dist}^2(\nabla v(x),\SO(2)) + \chi(x) {\rm dist}^2(\nabla v(x),\SO(2) F)\ dx\\
  & =  \ C \EE_{\rm elast}[\chi, v, B_{mR}] 
\end{split}
\end{equation}
for some constant $C=C(m, F)>0$. This completes the proof.
\end{proof}

We are now ready to give the key rigidity estimate. It is a variant of the two-well rigidity estimate from \cite{CS} and shows that we can find a sufficiently large rhombus such that we control the energy and the
change of length on all six connecting lines between the corner points of this rhombus both for the transformation and its inverse:
  
\begin{lemma}[Rigidity estimate] \label{lem:rigid domains} 
  Let $R>0$, $m\geq 1$, $\delta \in (0,\frac R{m})$. Then there are constants $\eta=\eta(\delta)>0$ and $C = C(\d,m,F)>0$ such that the
  following holds: Assume $(\chi, v)\in\AA_m$ satisfies
  $v\in C^1(\overline{B_R},\R^2)$,
  \begin{align}\label{as-ii}
    \NPL{\chi}{1}{B_R} \ \leq \ \eta R^2 \qquad 
    \text{and} \qquad %
    \NNN{\nabla \chi}{B_R} \ \leq \ \eta R.
  \end{align}
  Then there exist four points
  $\CC := \{ a, b, c, d \} \SUS B_{\frac{R}{m}}\subset \R^2$ with
  $|a-b| \sim R/m$ and $|c-d| \sim \delta R/m$, which form the end-points of a
  symmetric rhombus $T$ such that for all $x,y\in \CC$ and with the notation
    $M = \spt \chi$ we have the following properties
\begin{enumerate}
\item\label{it-nom} %
$\DS [x, y] \cap M = \emptyset$;
\item\label{it-noc} %
$\DS\EE_{\rm elast}[\chi,v,[x, y]] \ %
\leq \ \frac{C}{R} \EE_{\rm elast}[\chi, v,B_R]$;

\item\label{it-energyest} %
$\DS\int_{B_R} e_{\rm{elast}}(\chi, v) \frac{dz}{\dist(z,x)}  \ \leq  \  \frac{C}{R}  \EE_{\rm elast}[\chi, v, B_R]$.
\end{enumerate}
Furthermore, for $\chi_1 :=\chi \circ (v^{-1})$ we have
\begin{enumerate}[resume]
\item\label{it-nom-image} 
$\DS [v(x), v(y)] \cap v(M) = \emptyset$;  
\item\label{it-noc-image} %
   $\DS\EE_{\rm elast}'[\chi_1,v^{-1},[v(x), v(y)]] \ %
    \leq \ \frac{C}{R} \EE_{\rm elast}'[\chi_1, v^{-1},B_R]$;
\item \label{it-rot}
there exist $Q\in \SO(2)$ and $p\in\R^2$ such that
\begin{equation}
|v(x)-Qx-p|\leq C(\EE_{\rm elast}[\chi, v, B_R]^{\frac 12}+\eta^{\frac12}).
\end{equation}
 
  \end{enumerate}
  Finally, we have rigidity on all six segments
\begin{enumerate}[resume]
\item\label{it-urig} %
  $\DS\Big| 1 - \frac{|v(x)-v(y)|}{|x-y|} \Big| \ %
  \leq \ \frac{C}{R} \EE_{\rm elast}[\chi, v, B_R]^{\frac 12}$.
\end{enumerate}
\end{lemma}
\begin{proof}
  Without loss of generality, by scaling, we may assume that $R=m$ and
  $v(0)=0$. We further choose $\theta \in (0,1)$ sufficiently small to be
  determined below. We argue in several steps based on averaging-type arguments.

  \medskip
  
  \emph{Step 1: Identification of a symmetric cross satisfying \ref{it-nom} and
    \ref{it-noc}.} We first construct horizontal and vertical segments
  forming a ``cross'' satisfying \ref{it-nom} and \ref{it-noc}.  For
    $\d \in (0,\frac 12)$ and $r \in (-\d, \d)$ we define the horizontal line
    segment by $L_{\rm hor}(r) =[p_{-}(r),p_{+}(r)] \SUS B_1$ where
    $p_{\pm}(r):=(\pm \frac12,r)$.  We first show that if $\eta>0$ is sufficiently small, there exists a
  subset $E \subset (-\d,\delta)$ of volume fraction
  $1-\theta$ such that
\begin{equation}
L_{\rm hor}(r)\cap M = \emptyset \quad \text{and}\quad\EE_{\rm elast}[\chi,v, L_{\rm hor}(r)] \ \leq \ C\EE_{\rm elast}[\chi, v, B_1] \mbox{ for all } r\in E.
\end{equation}
Indeed, for some $C=C(\delta, \theta)>0$, we define
\begin{align}
  E \ %
  := \Big \{ r\in (-\d,\d): \NPL{\chi}{1}{L_{\rm hor}(r)} + \NNN{\nabla \chi}{L_{\rm hor}(r)} %
  \ \leq \  \theta \text{ and } \frac{\EE_{\rm elast}[\chi,v, L_{\rm hor}(r)]}{\EE_{\rm elast}[\chi, v, B_1]}  \ \leq \    C  \Big \}.
\end{align}
By Chebyshev's inequality and in view of \eqref{as-ii} we have
\begin{align}
  |(-\d,\d)\backslash E|
  \leq \ \frac{\eta}{\theta} + \frac{1}{C} \ %
  \leq \ (2\d) \theta
\end{align}
by choosing $\eta = \eta(\d,\theta)$ sufficiently small and $C = C(\d,\theta)$
sufficiently large. In particular, $|E| \geq 2\delta(1-\theta)$.  Now, since for
each $r\in E$, $\nabla \chi|_{L_{\rm hor}(r)}$ is a discrete measure and since
$\theta \in(0, 1)$, this implies $ \nabla \chi|_{L_{\rm hor}(r)} = 0$ for all
$r\in E$.  By definition of $E$ we then have $\chi_{|L_{\rm hor}(r)} = 0$ for
$r \in E$ (cf. \cite[p. 701]{KnuepferKohn-2011}).  This shows that outside of
volume fraction $\theta$, the horizontal segments $L_{\rm hor}(r)$ have the
properties \ref{it-nom}--\ref{it-noc}.

\medskip

Next, we repeat this argument along the vertical lines of the form $L_{\rm ver}(s)=[q_{-}(s),q_{+}(s)]$ with $q_{\pm}(s)=(s,\pm\d)$ for $s\in [-\frac12, \frac12]$. Also for this set, we analogously find a volume fraction $\tilde{E}\subset [-\frac12,\frac12]$ of size $1-\theta$ such that these vertical line segments satisfy \ref{it-nom}--\ref{it-noc}.
\begin{figure}[t]
    \centering 
\includegraphics{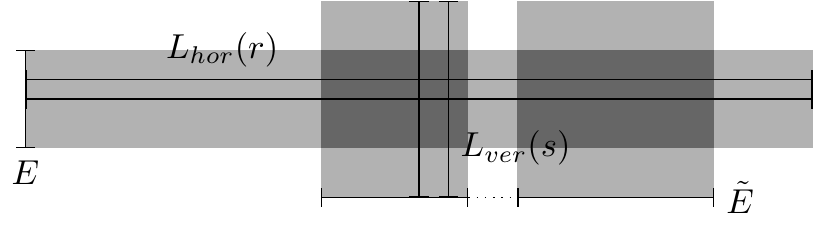}
    \caption{Grey rectangles represent sets of horizontal and vertical segments with the properties \ref{it-nom}--\ref{it-noc}.}
 \label{fig-cross} 
\end{figure}

\medskip 

Consider now the sets $\{L_{\rm hor}(r)\}_{r\in E}$ and
$\{L_{\rm ver}(s)\}_{s\in \tilde{E}}$ of all horizontal and vertical segments
with the properties \ref{it-nom}--\ref{it-noc}, respectively (see Fig.
\ref{fig-cross}). Let $o(s,r) =L_{\rm hor}(r) \cap L_{\rm ver}(s)$ be the
intersection point of the corresponding horizontal and vertical line.  The point
$o(s,r)$ divides both $L_{\rm hor}(r)$ and $L_{\rm ver}(s)$ into two segments
denoted by $L_{\rm hor}^+(r)$ and $L_{\rm hor}^-(r)$ (also $L_{\rm ver}^+(s)$
and $L_{\rm ver}^-(s)$).  Since $E$ and $\tilde{E}$ are sets of positive (close
to one) volume fractions, there exist $r_0\in E$ and $s_0\in \tilde{E}$ such
that $|L_{\rm hor}^+(r_0)|\sim |L_{\rm hor}^-(r_0)|$ and
$|L_{\rm ver}^+(s_0)|\sim |L_{\rm ver}^-(s_0)|$. Consequently, we choose
$L_{\rm hor}'$ and $L_{\rm ver}'$ such that
$o=L_{\rm hor}(r_0)\cap L_{\rm ver}(s_0)$ is the midpoint of $L_{\rm hor}'$ as
well as the midpoint of $ L_{\rm ver}'$. This can be done by (if necessary)
cutting exceeding parts of $L_{\rm hor}(r_0)$ and $L_{\rm ver}(s_0)$; we note
that such a modification preserves the conditions $|L_{\rm hor}'|\sim 1$ and
$|L_{\rm ver}'|\sim \d$.

\medskip

\emph{Step 2: Identification of a ``good'' rhombus.} Let $L_{\rm hor}'$ and
$L_{\rm ver}'$ be the segments forming a symmetric cross and satisfying
\ref{it-nom}--\ref{it-noc} as in the previous step. Let $\hat T$ be the
symmetric rhombus given by the convex hull of this cross. 
We denote by $\hat{T}_\rho$ the homothetically shrunken rhombus with the self-similarity factor $\rho \in (0,1]$ and the same center point. For
$\rho\in (\frac{1}{4},\frac{3}{4})=:I$, the diagonals (given by the
corresponding shortened line segments of the originally constructed cross) of
the resulting symmetric rhombi $\hat T_\rho$ also satisfy
\ref{it-nom}--\ref{it-noc} by construction, see Fig. \ref{fig-rhombi}. After using a Fubini argument as in Step 1, we obtain a subset $I_1$ of $I$ on which all sides of the rhombus fulfill the properties \ref{it-nom}--\ref{it-noc}.

\medskip

\begin{figure}[t]
    \centering 
\includegraphics{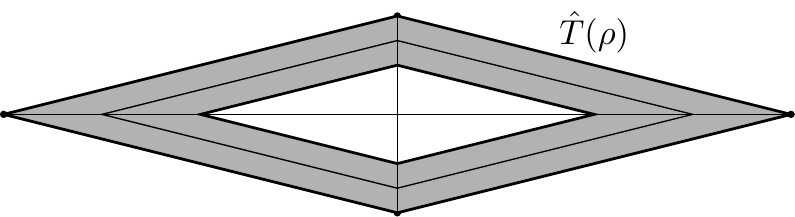}
    \caption{Sketch of a set of rhombi $\hat{T}(\rho)$,  $\rho\in (\frac{1}{4},\frac{3}{4})$.}
 \label{fig-rhombi} 
\end{figure}
Next, we seek to ensure that the properties
\ref{it-energyest}--\ref{it-noc-image} are also satisfied on the edges of some of
these rhombi. Invoking Lemma \ref{lem-regpoint} together with another averaging
argument, we obtain another set $I_2\subset I$ of positive volume fraction
satisfying \ref{it-energyest}.  On top of this, since $v$ is bi-Lipschitz and by
Lemma \ref{aux-lem} \ref{invVB}--\ref{invPB}, we can repeat Step 1 with the
functions $v^{-1}$ and $\chi_1$, the energy
$\EE'_{\rm elast}[\chi_1,v^{-1},B_m]$ and for the line segments $[v(x),v(y)]$,
where $x,y$ form the endpoints of the rhombi $\hat T_\rho$ for $\rho \in I$. Thus,
noting that by the bi-Lipschitz property of $v$, the length of the lines
  $[v(x),v(y)]$ is (up to a factor $m, m^{-1}$) comparable to that of $[x,y]$
  and after possibly enlarging the constant $C>0$, we obtain a subset $I_3$ of
$I$ with the properties \ref{it-nom-image}--\ref{it-noc-image}. By choosing the
intersection of these subsets of $I$, we arrive at a subset of $I$ with positive
volume fraction such that all sides of $\hat{T}_{\rho}$ fulfill
\ref{it-nom}--\ref{it-noc-image} for $\rho$ in this subset, provided $\eta>0$
is sufficiently small.

\medskip

By the Friesecke--James--Müller rigidity theorem \cite{FJM-2002} and Poincar\'e's
inequality, there exist $Q\in \SO(2)$ and $p\in \R^2$ such that for constants
$C_{\delta}, C_{F}>0$, we have
\begin{align}
  \|v(x)-Qx - p \|_{L^2(\hat{T}_\rho)}^{ 2} %
  &\leq \  C_\delta \|\nabla v - Q\|_{L^2(\hat{T}_\rho)}^{ 2}
    \ \leq \  C_\delta \| \dist(\nabla v, \SO(2)) \|_{L^2(\hat{T}_\rho)}^{2}\\
  &\leq \  C_\delta(\EE_{\rm elast}[\chi, v, B_m] + \dist(\SO(2)F, \SO(2)) |\hat{T}_\rho\cap M|) \\
  &\leq \  C_\delta(\EE_{\rm elast}[\chi, v, B_m] + C_F\eta).
\end{align}
Again, the use of a Fubini argument implies that there are many values of $\rho$
such that the resulting rhombi $\hat{T}_\rho$ are ``good'', in the sense that
all lines connecting the corner points of the rhombus satisfy the
properties \ref{it-nom}--\ref{it-noc-image} and that for some constant
$C=C(F,\delta)>0$ we have
\begin{equation}
  \|v(x)-(Qx+p)\|_{L^2(\p \hat{T}_\rho)}^{ 2}+\|\nabla v-Q\|_{L^2(\p\hat{T}_\rho)}^{ 2} \ %
  \leq C(\EE_{\rm elast}[\chi, v, B_m]+\eta).
\end{equation} 
Then by Sobolev's Embedding Theorem, we obtain 
\begin{equation}\label{bound-rot}
  \|v(x)-(Qx+p)\|_{L^\infty(\p \hat{T}_\rho)}^{ 2} \ %
  \leq \ C(\EE_{\rm elast}[\chi, v, B_m]+\eta).
\end{equation}
We choose one such ``good'' rhombus and denote it by $T$ and define its
endpoints as the points $\CC := \{ a,b,c,d \}$ (see Fig. \ref{fig-rhombus}).
Since $v$ is a continuous function, we obtain from inequality \eqref{bound-rot}
\begin{equation}
  |v(x)-(Qx+p)|^{ 2} \ %
  \leq \ C(\EE_{\rm elast}[\chi, v, B_m]+\eta) %
  \qquad\text{for }x\in \CC.
\end{equation}
As a consequence, by construction the properties \ref{it-nom}-- \ref{it-rot} are satisfied for these endpoints.

\begin{figure}[t] 
    \centering 
\includegraphics{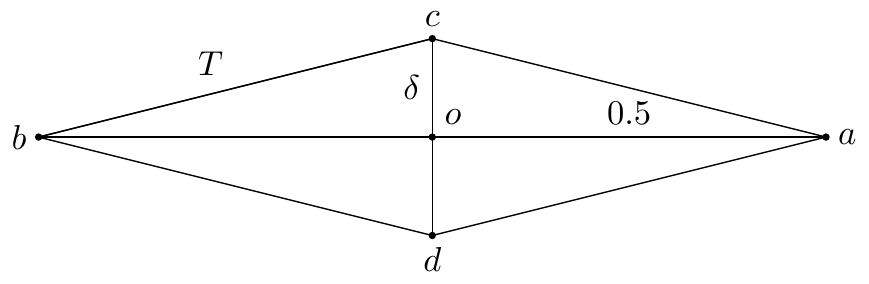}
    \caption{Sketch of the rhombus $T$ constructed in Lemma \ref{lem:rigid domains}.}
 \label{fig-rhombus} 
\end{figure}

\medskip

\emph{Step 3: Proof of \ref{it-urig}.}  By the fundamental theorem of
calculus, for any $x,y \in \CC$ we have 
\begin{align}\label{ineq*}
  |v(x)-v(y)| \ %
  &\leq  \ \int_{[x,y]} |\nabla v|  \ %
    \leq \   |x-y| + \int_{[x,y]} \dist(\nabla v,\SO(2))  \\
  &\stackrel{\ref{it-noc}}{ \ \leq \ }  \ |x-y| + C \EE_{\rm elast}[\chi,v,B_1]^{\frac 12}.
\end{align}

Now, we apply the same argument to $v^{-1}(v(x))-v^{-1}(v(y))$ with
$x,y \in \CC$.  Thus, in view of Lemma \ref{aux-lem}, for a constant
$C=C(\d, m, F)>0$ we obtain
\begin{align}\label{ineq**}
  |x-y| \ %
  &\leq  \ \int_{[v(x),v(y)]} |\nabla v^{-1}(z)|  \\ 
   &\leq \ |v(x)- v(y)| + \int_{[v(x),v(y)]} \dist(\nabla (v^{-1}),\SO(2))  \\
   &\stackrel{\ref{it-nom-image},\ref{it-noc-image}}{\leq} |v(x)- v(y)|+C\EE_{\rm elast}'[\chi_1, v^{-1},B_1]^{\frac 12}\\
   &\leq |v(x)- v(y)|+C\EE_{\rm elast}[\chi, v,B_m]^{\frac 12}.
\end{align}
Combining inequality \eqref{ineq*} and inequality \eqref{ineq**}, we obtain the desired estimate \ref{it-urig}. This completes the proof of the lemma. 
\end{proof}

\section{A Lower Bound for the Elastic Energy}
\label{sec:lower}

In this section we prove a local lower bound by exploiting the rigidity
argument from Lemma \ref{lem:rigid domains} and the ideas from the proof of
Lemma 2.3 in \cite{CS}.  This local lower bound provides a geometrically nonlinear variant of the central lower bound from Proposition 3.1 in \cite{KnuepferKohn-2011}. In Section \ref{sec:proof_lower_bound} we will combine it with a covering argument as in \cite{KnuepferKohn-2011} which will imply the lower bound of Theorem \ref{thm-unscaled}.

\begin{proposition}[Lower bound on elastic energy] \label{prp-stray} %
  There is $\eta > 0$ such that for any $R > 0$ the following holds:
  Suppose that $(\chi, v)\in\AA_m$ satisfies
  \begin{align} \label{s-small} %
    \NPL{\chi}{1}{B_R} \ \leq \ \eta R^2 %
    && \text{ and } && %
                       \NNN{\nabla \chi}{B_R} \ \leq \ \eta R.
  \end{align}
  Then there are constants $\alp= \alp(m) \in (0,1)$ and $C = C(F,\alpha,m)>0$
    such that
  \begin{align} \label{bound} %
    \EE_{\rm elast}[\chi,v,B_R] \ %
    \geq \ \frac {C}{R^2} \|\chi\|_{L^1(B_{\alpha R})}^2.
  \end{align}
\end{proposition}
\begin{proof}
  This result essentially follows from an application of a variant of the
  two-well rigidity result from \cite{CS}.  Here there are slight adaptations in Steps 1 and 2 in the proof due to the choice of our energies (full gradient control in \cite{CS} vs.
  our phase-indicator energies), while Steps 3 and 4 then follow essentially without changes as in \cite{CS}. For self-containedness, we repeat the
  argument for Proposition \ref{prp-stray}. 
   
  By scaling we can assume $R=m$ and by the approximation
  results in \cite{DaneriPratelli-2013} for bi-Lipschitz functions we can
  further assume that $v \in C^1(\overline{B_m})$.

\medskip

Following the argument in \cite{CS} in our proof we will construct a rhombus $T$ with $B_\alp \subset T \SUS B_1$ and show that the corresponding estimate
\eqref{bound} holds for $T$ replaced by $B_\alp$ for some $\alp > 0$. We write
$\mu := \NPL{\chi}{1}{B_{\alpha}}$ and $\eps := \EE_{\rm
  elast}[\chi,v,B_m]$. Moreover, we note that, without loss of generality, we
can assume  
\begin{align} \label{eps-small} %
  \eps^{\frac12} < \eta.
\end{align}
Indeed, if $\eps$ is large
e.g. $\eps^{\frac12}\geq \eta$, then by assumption we have
$\|\chi\|_{L^1(B_1)}\leq \eta\leq \eps^{\frac12}$. Then inequality
\eqref{bound} follows immediately.

\medskip

  \textit{Step 1: Construction of a ``good'' rhombus.}  Since $F \neq \Id$ and
  $\det F = 1$ after a rotation of coordinates, we may assume that
  $|Fe_1|<1$. Hence, there exists $\delta>0$ such that
  \begin{align}\label{eq2.27}
    |F\xi|<1-2\delta\quad\text{for all }\xi\in\R^2 %
    \qquad\text{with }|\xi|=1 \text{ such that } |\xi-e_1|<2\delta.
  \end{align}
  Without loss of generality we can assume that $\d \in (0, 1)$ so
    that the conditions of Lemma \ref{lem:rigid domains} with $R = m$ are
    satisfied. We then consider a rhombus $T$ with corner points
    $\CC := \{ a, b, c, d  \}$ as obtained in Lemma \ref{lem:rigid domains}, see
    Fig. \ref{fig-rhombus}. Since $|c-d|\sim \delta$ and $|a-b| \sim
  1$, in particular,
  \begin{align}\label{short-segm}
    |F(p-t)| < |p -t|(1 -\delta) \qquad  \mbox{ for all } p \in \{ a, b \}, t \in [c,d].
  \end{align}
  By Lemma \ref{lem:rigid domains} we further have the properties
  \ref{it-nom}--\ref{it-urig} for this rhombus.  

  \medskip

  \textit{Step 2: We claim that there exists $Q\in \SO(2)$ and $p \in \R^2$ such
    that $v(x)$ is close to $Q x + p$ for any point $x\in \CC$ up to an
    error of order $\eps^{\frac12}$.}  Indeed, by Lemma \ref{lem:rigid domains}
  \ref{it-urig} the six lengths $|x-y|$ for $x,y\in \CC$ are preserved by
  $v$ up to errors of order $\varepsilon^{\frac 12}$. This implies that there
  are two isometries $x\to Q_j x + p_j$ with $Q_j \in \mbox{O}(2)$, $p_j\in \R^2$ and
  $j\in \{1,2\}$ such that for the constant $C=C(\d,m,F)>0$ from Lemma \ref{lem:rigid
    domains} we have
    \begin{align}
    \label{eq:iso}
    \begin{split}
      &|v(x)-(Q_1 x+p_1 )| \ \leq \  C\eps^{\frac 12}\qquad\text{for }x\in\{b,c,d\}  \text{ and } \\
      &|v(x)-(Q_2 x+p_2 )| \ \leq \ C\eps^{\frac 12}\qquad\text{for
      }x\in\{a,c,d\}.
      \end{split}
    \end{align}
    It remains to argue that $Q_1, Q_2 \in \SO(2)$ and 
    $p_1, p_2 \in \R^2$ can be chosen to be equal, respectively.

    \medskip
    
    We first argue that $Q_j \in \SO(2)$. In \cite{CS} this follows from the
    second gradient control and the pointwise estimates in the endpoints of the
    rhombus. Lacking the control of the full gradient, we here vary the argument
    slightly.  The use of Lemma \ref{lem:rigid domains} \ref{it-rot} and the
    triangle inequality implies that for some constant $C \ = C(F,\delta,m)>0$
    and for $Q\in \SO(2)$, $p\in \R^2$ we have
    \begin{equation}
      |Qx+p-(Q_1x+p_1)|\leq C (\eps^{\frac12}+\eta^{\frac12})\text{ for }x\in\{b,c,d\}.
    \end{equation}
    For $\eta \in (0,1)$ (depending on $\delta>0$) and $\eps>0$ sufficiently
    small, this yields a contradiction, if $Q_1 \in \mbox{O}(2)\setminus \SO(2)$. Similarly, we
also obtain that $Q_2 \in \SO(2)$.  Moreover, since the triangles $\Delta_{cbd}$
with vertices $c, b, d$ and $\Delta_{acd}$ with vertices $a, c, d$ share a
common line, we have that $Q_1$ can be chosen equal to $Q_2$ and that
$p_1 = p_2$. A normalization further allows us to suppose that $p_1=p_2=0$ and
$Q_1 = Q_2 = \Id$. As a consequence, we may assume that
\begin{align}
  \label{eq:isometry}
  &|v(x)-x| \ \leq \  C\eps^{\frac 12}\qquad\text{for }x\in\{a,b,c,d\} \text{ and for some constant }C=C(F,\d,m).
\end{align}

  \medskip

  \textit{Step 3: Smallness estimate for $N$:} As in \cite{CS}, we claim that
  \begin{align} \label{N-est} %
    |N\cap T| \ \leq \ C\eps^{\frac 12} \text{ for some constant } C=C(F,\d,m),
  \end{align}
where the set $N$ denotes the region where the gradient is closer
    to the well $\SO(2)F$ than to the parent gradient, i.e.
    \begin{align}\label{eq2.24}
      N \ := \ \big\{x\in  B_1: \dist(\nabla v(x),\SO(2)F)< \dist(\nabla v(x),\SO(2)) \ \big \}.
    \end{align}
To this end, we use the upper length bounds on $v(t)$, i.e. the fact that $v$ is
essentially not length increasing.  Let $t$ be any point of $[c, d]$. By the
fundamental theorem of calculus and Lemma \ref{lem:rigid domains} \ref{it-noc}
we then get for some constant $C=C(\d)>0$
  \begin{align}
    |v(c)-v(t)| \ \leq \ |c -t|+\int_{[c,t ]}\dist(\nabla v,\SO(2)) \ \leq \ 
    |c-t| + C\eps^{\frac 12}.
  \end{align}
  Combining this with the triangle inequality and the bound \eqref{eq:isometry}
  applied to $x=c$, we obtain
\begin{align}\label{eq:distance}
  |c-v(t)| \ %
  \lupref{eq:isometry}{\leq}|c-t|+C\eps^{\frac 12}  %
  \text{ for all $t \in [c,
  d]$} \text{ and some constant } C=C(F,\d, m)>0.
\end{align}
We note that in view of \eqref{eps-small} and for
  $\eta = \eta(\d)$ sufficiently small we can assume that
  \begin{align} \label{eps-ssmall} %
  C\eps^{\frac 12} \ < \ \frac 12 |c-d|.
  \end{align}
Next, we seek to use this to deduce lower bounds on $|a- v(t)| + |b - v(t)|$
for $t\in [c,d]$ as above. To this end, we observe that in view of
  \eqref{eps-ssmall}, the minimization problem
\begin{align}
  \min\left\{ |a-t'| + |b-t'| \ : \ t' \in B_{r_{c,t}}(c) \text{ with } r_{c,t} := |c- v(t)|\right\}
\end{align}
is attained on the line $[c,d]$ and is solved by
$t^{\ast}:= t-r \frac{c-d}{|c-d|}$ for some $r$ with $0<r<C\eps^{\frac{1}{2}}$. Here, the error bound for $r$ is a consequence of
\eqref{eq:distance}.  Using $v(t)$ as a competitor and inserting the bound for
$r_{c,t}$ implies
\begin{equation}
  |a- v(t)| + |b - v(t)| \ %
  \geq |a- t^{\ast}| + |b - t^*| \ %
  \geq \  |a - t| + |b - t| - C\eps^{\frac 12}
  \qquad %
  \text{for all $t \in [c,
      d]$.}
\end{equation}

Using again \eqref{eq:isometry} now for $x=a$ and $x=b$, we infer the following
lower bound on the length deformation for points $t\in [c,d]$:
\begin{align}\label{eq2.30}
  |v(a) - v(t)| + |v(b) - v(t)|  \ \geq \  |a - t| + |b - t| - C\eps^{\frac 12}.
\end{align}
We complement this with an upper bound on the length deformation along the
segments $[a, t]$ and $[t, b]$, obtained by means of the fundamental theorem.
In view of \eqref{short-segm} and using
\begin{equation}
  |\partial_{\xi}v| \ \leq \  1 + (|F\xi|-1)\chi_N + \dist(\nabla v, K) \qquad %
  \text{for $\xi \in \Big \{ \frac{a-t}{|a-t|}, \frac{b-t}{|b-t|} \Big \}$ and any $t\in [c,d]$},
\end{equation}
where $K := \SO(2) \cup \SO(2)F$ and $\chi_N$ denotes the
characteristic function of the set $N$ (cf. \eqref{eq2.24}) we get
\begin{equation}
  |v(p) - v(t)| \ \leq \  |p - t| +\int_{[p,t ]}\dist(\nabla v, K)\
 -\delta\int_{[p,t ]}\chi_N\  \qquad %
  \text{for $p \in \{ a, b \}$} .
\end{equation}
Subtracting these
estimates from \eqref{eq2.30} we arrive at
\begin{align}
  \int_{[a,t]\cup[t,b]}\chi_N \ \leq \ \delta^{-1}\int_{[a,t]\cup[t,b]}\dist(\nabla v, K)\  + C\delta^{-1}\eps^{\frac 12} \qquad %
  \text{for any $t\in [c,d]$}.
\end{align}
We integrate over all $t\in [c, d]$ and change variables from $(x_1,t_2)$ to
$(x_1,x_2)$ by the transformation
$\Psi(x_1,t_2) = (x_1,t_2 (1- \frac {x_1}{a_1}))$ (where $t=(0,t_2)$,
$a=(a_1,0)$) and $\Phi = \Psi^{-1}$ to obtain an integration over the rhombus
$T$. More precisely, denoting by $J_\Phi(x)$ as the Jacobian determinant of
  the transformation $\Phi$, we infer
\begin{align}\label{eq*}
  \int_T\chi_N J_\Phi\  \ \leq \  C_{\d}\int_T\dist(\nabla v, K) J_\Phi \ + C\eps^{\frac 12}.
\end{align}
Since $|J_\Phi| \sim $
$\dist(x, \{a, b\})^{-1}$, and thus, in particular,
$J_\Phi \geq 1$, in the left-hand side we can simply drop $J_\Phi$. For the
right hand side we invoke Lemma \ref{lem:rigid domains} \ref{it-energyest} which
concludes the argument for \eqref{N-est}.

\medskip

\textit{Step 4: Conclusion.} Last but not least, it remains to estimate
$|B_{\alpha}\cap M|$. For $\alp :=\frac{\d}{4}$ we have $B_\alp \SUS T$. By definition of $N$ and the triangle inequality we then have
\begin{align}\label{N-ineq}
  \int_{B_\alp}{\dist}^2(\nabla v, \SO(2))\ 
  & \leq  \ 2\EE_{\rm elast}[\chi,v,B_\alp] + 2\int_{B_\alp \cap N} \|\Id-F\|^2 \ \\ 
  &\stackrel{\eqref{def-eps}}{ \ \leq \ } 2\eps+2\|\Id-F\|^2|N\cap T| \
    \lupref{N-est}\leq \ C\eps^{\frac 12}.
\end{align}

By \cite[Theorem 3.1]{FJM-2002}, we have for some $W,Q\in L^{\infty}(B_{\alpha}, \SO(2))$,
\begin{equation}\label{M-dist}
\begin{split}
\|{\dist}&(\nabla v, \SO(2))\|_{L^2(B_\al)}
\ \gtrsim \ \left(\int_{B_\al}\chi\|\nabla v-W\|^2\right)^{\frac12}\\& \ \geq
\
\left(\int_{B_\al}\chi\|QF-W\|^2\right)^{\frac12}-\left(\int_{B_\al}\chi\|\nabla
  v-QF\|^2\right)^{\frac12}\\& \ \geq \ \dist(\SO(2)F,\Id)|M\cap
B_\al|^{\frac12}-\left(\int_{B_\al}\chi \text{dist}^2(\nabla v,
  \SO(2)F)\right)^{\frac12}.
\end{split}
\end{equation} 
Here $Q\in L^{\infty}(B_{\alpha},\SO(2))$ is such that
$\dist(\nabla v, \SO(2)F)=\|\nabla v-QF\|$ for almost every $x\in
B_{\alpha}$. Hence, we obtain 
\begin{equation}
\begin{split}
|M\cap B_\al|&\stackrel{\eqref{M-dist}}{ \ \leq\ } C \int_{B_\al}{{\dist}}^2(\nabla v, \SO(2))+C\int_{B_\al}\chi\text{dist}^2(\nabla v, \SO(2)F)\\&\stackrel{\eqref{N-ineq}}{ \ \leq \ } C\eps^{\frac12} \text{ for some constant } C=C(F,\d, m).
\end{split}
\end{equation}

This is the assertion of the theorem.
\end{proof}

\section{Proof of Theorem \ref{thm-unscaled}}

We are ready to give the proof of Theorem \ref{thm-unscaled}. We split it into
two parts and first discuss the lower bound and then provide a matching upper
bound construction.

\subsection{Proof of the lower bound in Theorem \ref{thm-unscaled}}
\label{sec:proof_lower_bound}

In this section, we provide the proof of the lower bound.  To this end, we first
observe that in the small volume regime this directly follows from the isoperimetric
inequality. It thus suffices to consider the large volume regime $\mu\geq 1$.
Although the proof follows the localization argument as in
\cite{KnuepferKohn-2011}, for the convenience of the reader, we briefly recall
its proof.

\begin{proof}[Proof of Theorem \ref{thm-unscaled}, lower bound]

  \emph{Step 1: Strategy.}  We argue by a localization and covering argument,
  seeking to invoke Proposition \ref{prp-stray}. We consider a suitably chosen
  countable family of balls $\{B_{R_i}(x_i)\}_{i=1}^\infty$ covering
  $M:= \spt \chi$ (see Step 2 below). By a Vitali covering argument, we may
  assume that $\{B_{R_i/5}(x_i)\}_{i=1}^\infty$ are pairwise disjoint. Then, we
  can localize the energy as follows:
\begin{align}
\EE[\chi, v]   \ \geq \  \sum_{i=1}^\infty \EE_{R_i/5}(x_i),
\end{align} 
where $\EE_{R_i}(x_i):= \EE[\chi, v, B_{R_i}(x_i)]$. Now, if we could bound
$\EE_{R_i/5}(x_i)$ from below in terms of $|M\cap B_{R_i}(x_i)|^{2/3}$ we could
conclude the argument
\begin{align}
\EE[\chi, v]  \ \geq \  \sum_{i=1}^\infty \EE_{R_i/5}(x_i)
 \ \gtrsim \  \sum_{i=1}^\infty |M\cap B_{R_i}(x_i)|^{2/3}  \ \gtrsim \  \mu^{2/3}.
\end{align}
It thus remains to argue that
\begin{align}
\EE_{R_i/5}(x_i)  \ \gtrsim \  |M\cap B_{R_i}(x_i)|^{2/3}.
\end{align}
We split this into two steps: Following \cite{KnuepferKohn-2011}, we prove that
\begin{align} \label{Ineq1} %
  \DS|M\cap B_{\alpha R_i/5}(x_i)|  \ %
  &\gtrsim \  |M\cap B_{ R_i}(x_i)|; \\
  \DS\EE_{R_i/5}(x_i) \
  &\gtrsim \ |M\cap B_{\alpha R_i/5}(x_i)|^{2/3}, \label{Ineq2}
\end{align}
where $\alpha>0$ is the constant from Proposition \ref{prp-stray} and for
suitably chosen balls $B_{R_i}(x_i)$. We note that all the estimates in this
proof may depend on the constant $\alp$.

\medskip

\emph{Step 2: Choice of radii and center points $x_i$.}  To this end, without
loss of generality, we may assume that all $x \in M$ are points of density one
of $M$.  Now for any $x\in M$ we set
  \begin{align}\label{eq2*}
    R(x):=\inf\left\{r:r^{-2}|M\cap B_r(x)| \ \leq \  \eta_0\min\{1, |M\cap B_r(x)|^{-\frac13} \}\right\},
  \end{align}
  where $\eta_0$ is sufficiently small constant, which will be fixed later on. 
 By continuity in $r$ and by considering the limit $r\rightarrow \infty$,
we infer that $R(x)\leq\frac{\mu^{\frac23}}{\sqrt{\eta_0}}$.
Therefore, $R(x)$ is uniformly bounded in terms of $\mu$ and the defining infimum actually is a minimum.
  Similarly as in \cite{KnuepferKohn-2011}, we note that $R=R(x)$ satisfies one of
  the following conditions: Either
  \begin{align}\label{eq4*}
    |M\cap B_R(x)| \ \leq \  1 \text{  and }|M\cap B_R(x)|=\eta_0R^2
  \end{align}
  or
  \begin{align}\label{eq5*}
    |M\cap B_R(x)|> 1 \text{  and }|M\cap B_R(x)|=\eta_0|M\cap B_R(x)|^{-\frac13}R^2.
  \end{align}
Obviously, $M$ is covered by
  $\cup_{x\in M}B_{R(x)}$. Since the radii $R(x)$ are uniformly bounded, by Vitali's covering lemma, there is an at most countable subset of points $x_i\in \R^2$ such that the
  balls $\{B_{R_i/5}(x_i)\}_{i=1}^\infty$ are pairwise disjoint while $M$ is still covered by the balls
  $\{B_{R_i} (x_i)\}_{i=1}^\infty$. This yields the balls and radii from Step 1.
  
\medskip

\emph{Step 3: Proof of the estimate \eqref{Ineq1}.}
By the definition of $R$, we obtain the following statements:
if $|M\cap B_{R_i} (x_i)|\leq 1$ and $|M\cap B_{\alpha R_i/5}(x_i)|\leq 1$, then
\begin{align}\label{eq9*}
  \frac{|M\cap B_{\alpha R_i/5}(x_i)|}{(\al R_i/5)^2}\stackrel{\eqref{eq2*}}{ \ \geq \ }\frac{|M\cap B_{R_i} (x_i)|}{R_i^2} \lupref{eq4*}=\eta_0.
\end{align} 
Analogously, if $|M\cap B_{R_i} (x_i)|> 1$ and $|M\cap B_{\alpha R_i/5}(x_i)|> 1$, then
\begin{align}
  \frac{|M\cap B_{\alpha R_i/5}(x_i)|^{\frac43}}{(\al R_i/5)^2}\stackrel{\eqref{eq2*}}{ \ \geq \ }\frac{|M\cap B_{R_i} (x_i)|^{\frac43}}{R_i^2}\lupref{eq5*}=\eta_0.
\end{align} 
Finally, if $|M\cap B_{\alpha R_i/5}(x_i)|\le 1< |M\cap B_{R_i} (x_i)|$, then
\begin{align}\label{eq10*}
  \frac{|M\cap B_{\alpha R_i/5}(x_i)|}{(\al R_i/5)^2}>\eta_0 \lupref{eq5*}=\frac{|M\cap B_{R_i} (x_i)|^{\frac43}}{R_i^2}>\frac{|M\cap B_{R_i} (x_i)|}{R_i^2}.
\end{align}
The last three obtained estimates together yield bound \eqref{Ineq1}.

\medskip

\emph{Step 4: Proof of the estimate \eqref{Ineq2}.}
Here, we distinguish three cases:
Firstly, we assume that case \eqref{eq4*} holds. Since the density of the minority phase is much smaller than one in $B_{R_i/5}(x_i)$, the use of the isoperimetric inequality implies 
\begin{align}\label{eq11*}
  \int_{B_{R_i/5}(x_i)}|\nabla \chi| \ \gtrsim \ |M\cap B_{R_i/5}(x_i)|^{\frac 12}\stackrel{\eqref{Ineq1}}{ \ \sim \  }|M\cap B_{R_i} (x_i)|^{\frac 12}  \stackrel{\eqref{eq4*}}{ \ \gtrsim \ }|M\cap B_{R_i} (x_i)|^{\frac23}.
\end{align}

Secondly, we suppose that case \eqref{eq5*} and 
\begin{align}\label{eq12*}
  \int_{B_{R_i/5}(x_i)}|\nabla\chi| \ \gtrsim \  R_i \text{ hold}.
\end{align}
Since $R_i \sim  |M\cap B_{R_i} (x_i)|^{\frac23}$, we derive
\begin{align}
  \EE_{R_i/5}(x_i) \ \gtrsim \ \int_{B_{R_i/5}(x_i)}|\nabla\chi| \stackrel{\eqref{eq12*}}{ \ \gtrsim \ }|M\cap B_{R_i} (x_i)|^{\frac23}.
\end{align}
Lastly, we assume that case \eqref{eq5*} and 
\begin{align}
  \int_{B_{R_i/5}(x_i)} |\nabla\chi| \ \ll \  R_i \text{ hold},
\end{align}
where $\ll$ means that this estimate requires a small universal constant.

Here, choosing $\eta_0$ small enough, the
assumptions of Proposition \ref{prp-stray} are fulfilled on
$B_{R_i/5}(x_i)$. The use of this proposition results in
\begin{align}
  \EE_{R_i/5}(x_i) \ \gtrsim \ \frac{|M\cap B_{\alpha R_i/5}(x_i)|^2}{(\al R_i/5)^2}\stackrel{\eqref{Ineq1}}{ \ \sim \  }\frac{|M\cap B_{R_i} (x_i)|^2}{R_i^2} \ \sim \  |M\cap B_{R_i} (x_i)|^{\frac23},
\end{align}
as $R_i \sim  |M\cap B_{R_i} (x_i)|^{\frac23}$. Then, inequality
\eqref{Ineq2} follows from the above estimates, which concludes a proof of the lower bound in Theorem
\ref{thm-unscaled}.
\end{proof}

\subsection{Proof of the upper bound of Theorem \ref{thm-unscaled}}

We next give the proof of the upper bound in Theorem \ref{thm-unscaled}.  For
this, we give an explicit construction for an optimal configuration.  It
suffices to consider the case $\mu\geq 1$, since the case $\mu\leq 1$ follows by
simply considering $v(x)=x$ and $\chi = \chi_B$ where $B$ is a ball with
$|B| = \mu$. The estimate then follows by using the isoperimetric inequality and
noting that $0<\mu\leq \mu^{\frac{1}{2}}$ if $\mu\leq 1$.  We note that similar
constructions are well established (e.g.  \cite{Khachaturyan-1982}). An upper
bound in the setting of geometrically linear elasticity has also been given in
\cite{KnuepferKohn-2011} for the geometrically linearized theory. We provide an
analogous construction for the geometrically nonlinear case and check that the
solutions are within our class of admissible functions.  We first note that, by
a rotation (see Lemma \ref{lem-fform} for more details), we can assume that
\begin{align}
  F \ %
  = \ \Id+\nu\otimes e_2 \text{ for some $\nu = \begin{pmatrix} \nu_1\\0 \end{pmatrix}  \in \R^2$ } .
\end{align}
In particular $e_2$ is one of the twinning directions for stress--free laminates between $Fx$ and $x$.

\medskip

\begin{figure}[t] 
  \centering 
  \includegraphics{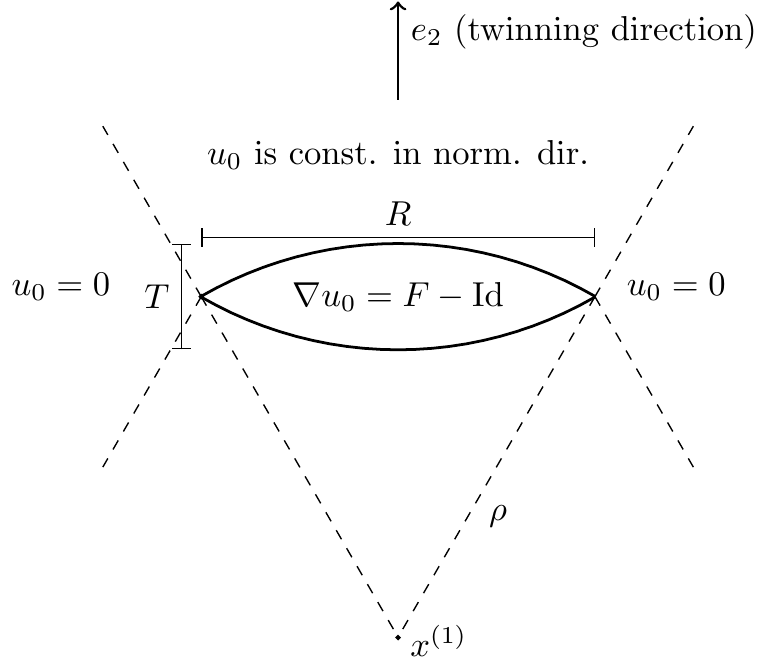}
  \caption{Sketch of the construction of $u_0$ and $v=\omega_R u_0+x$.}
  \label{fig-constr} 
\end{figure}
As in \cite{KnuepferKohn-2011} we consider an inclusion which approximately has
the shape of a thin disc $Q_{T,R}$ with diameter $R$ and thickness $T$ where
$T \ll R$.  The disc is oriented such that the two large surfaces are aligned
with the $e_2$ twinning direction. To be more precise, let
$x^{(1)},x^{(2)} \in \R^2$ such that $x^{(1)}=-x^{(2)}$ on the axis $x_1=0$ with
distance $d := |x^{(1)}-x^{(2)}|$. We define $\chi$ by
\begin{align}
  \chi \ := \ \chi_{Q_{T,R}}, \qquad \text{where } Q_{T,R}:=B_\rho(x^{(1)})\cap
B_\rho(x^{(2)}),
\end{align}
where $Q_{T,R}$ is the lens with thickness of order $T$ and diameter of order
$R$ given by the intersection $B_\rho(x^{(1)})\cap B_\rho(x^{(2)})$ for some
suitable $\rho = \rho(R,T)>0$. We choose $T$ such that it fulfills the volume
constraint \eqref{volume}, i.e. $|Q_{T,R}| = \mu$ and in particular,
$RT\sim \mu$.

\medskip

We next define $u_0 :\R^2 \to \R^2$ such that $u_0(x)= (F - \Id) x$ in
$Q_{T,R}$. Furthermore, outside $Q_{T,R}$, $u_0$ is constant on all lines which
are normal to the surface $\p Q_{T,R}$. Finally, $u_0 =0$ in the remaining area
which is neither in $Q_{T,R}$ nor reached by any of these lines. The function is
sketched in Fig. \ref{fig-constr}. Furthermore, let
$\ome_R\in C^\infty(\R,[0,\infty])$ be a cut-off function with $\ome_R(\xi)= 1$
for $|\xi| \ \leq \ R$ and $\ome_R(\xi) = 0$ for $|\xi| \ \geq \ 2R$ with
$|\nabla\ome_R|\leq\frac{C}{R}$ for fixed $C>0$.  We then define
$v :\R^2 \to \R^2$ by
\begin{equation}\label{up-func}
  v(x)  \ %
  := \ (\ome_R u_0)(x)+x.
\end{equation}

\textit{Estimates:} By construction we have
\begin{align} \label{u0-est} %
  \|\nabla u_0(x)\| \ %
  \lesssim \ \begin{cases}
    1 & \text{ for }x\in Q_{T,R},\\
    TR^{-1} & \text{ for } x\notin Q_{T,R}
  \end{cases}
\end{align}
and $\NIL{u_0}{\R^2} \lesssim T$. Since $\nabla v = F$ in
$Q_{T,R}$ and $\nabla v = \Id$ in $B_{2R}^c$ we get
\begin{align}
  \EE_{\rm elast}[\chi,v] \ %
  &\leq \ \int_{B_{2R}\backslash Q_{T,R}} {\rm dist}^2(\nabla v,\SO(2))  \ %
    \lesssim \ \int_{B_{2R}\backslash Q_{T,R}}\NNN{\nabla v-\Id}{}^2 \\
  &\lesssim \ \int_{B_{2R}\backslash
    Q_{T,R}} \NNN{\nabla \omega_R\otimes u_0}{}^2+\int_{B_{2R}\backslash
    Q_{T,R}}\NNN{\omega_R \nabla u_0}{}^2.
\end{align}
By \eqref{u0-est}, since $RT \sim \mu$ and also including the interfacial part
of the energy we obtain
\begin{align}
\label{eq:elast_bd}
\begin{split}
  \EE[\chi, v]& \ \lesssim \ \int_{\R^2} |\nabla \chi| + \int_{B_{2R}\backslash
    Q_{T,R}} \frac {T^2}{R^2}  \ %
  \lesssim \ R + T^2 \ %
  \lesssim \ R + \frac{\mu^2}{R^2}.
\end{split}
\end{align}
The asserted upper bound then follows with the choice $R \sim
\mu^{2\slash3}$.

\medskip

\textit{Admissibility:} We need to check that our construction satisfies
$(\chi, v) \in \AA_m$. In fact, it is enough to check this condition for
$\mu = \mu(F)$ and correspondingly $R \sim \mu^{2/3}$ sufficiently large. We
first note that $\chi \in BV(\R^2, \{ 0 , 1\})$ and
$\NI{\nabla v} \leq C \|F\|$. We next consider $v$ locally in the different regions defining it and show that in each of these $v^{-1}$ exists and $\NI{(\nabla v)^{-1}}{} \leq m$.  To this end, we recall that
\begin{align}
  \nabla v \ = \ \nabla \omega_R\otimes u_0 + \omega_R \nabla u_0 + \Id.
\end{align}
For $x \nin B_{2R}$ we have $\nabla v = \Id$ and
$\NNN{(\nabla v)^{-1}}{} =\sqrt{2}\leq\|F\|$. For $x \in Q_{T,R}$ we have
$\nabla v = F$ which implies $\NNN{(\nabla v)^{-1}}{} \leq \NNN{F^{-1}}{}$. It
hence remains to estimate $(\nabla v)^{-1}(x)$ for $x \in B_{2R} \BS Q_{T,R}$.
Let $(b_1(x),b_2(x))$ for $x \in B_{2R} \BS Q_{T,R}$ be the mathematical
positive oriented basis where $b_2(x)$ is the direction of the lines in
$B_{2R} \BS Q_{T,R}$ where $u_0$ is constant and with sign convention
$b_2(x) \cdot e_2 > 0$. By construction we then have
$|b_i(x) - e_i| \leq \OO(\frac TR)$ for $i = 1,2$. Since
$\nabla u_0(x) b_2(x) = 0$ we hence get
$|\nabla u_0 e_2| \leq \frac{C \|F\| T}{R}$. Since $(F-\Id)e_1 = 0$ we also have
$|\nabla u_0(x) b_1(x)| \leq \frac{C \|F\| T}{R}$. Together, this yields
$\NNN{\nabla u_0}{} \leq \frac {C \|F\| T}{R}$. Since $|\ome_R| \leq 1$ and
$|\nabla \ome_R| \lesssim \frac 1R$, this yields
\begin{align}
  \|\nabla \omega_R\otimes u_0 + \omega_R \nabla u_0\| \ %
  \leq \  \frac {C \|F\| (1+T)}{R}.
\end{align}
In particular,
\begin{equation}
  \NNN{\nabla v - \Id}{} \leq \frac {C \|F\| (1+T)}{R}\leq C \|F\|\mu^{-\frac13}
\end{equation}
as $R\sim\mu^{\frac23}$, $T\sim\mu^{\frac13}$ and $\mu\geq 1$. As a consequence,
a Neumann series argument implies that
\begin{equation}
\|(\nabla v)^{-1}(x)\|\leq C(1+\|\nabla v(x)- \Id\|) \leq C\|F\|
\end{equation}
and for $R = R(\|F\|)$
sufficiently large.  

Last but not least, we argue that with the observations for $\nabla v$ from above, we obtain that $v$ is globally invertible. To this end, it suffices to prove that $v$ is injective. Assuming that for some $x,y\in \R^2$ we have that $v(x)=v(y)$, the fundamental theorem yields that
\begin{align}
0 = \left( \int\limits_{0}^1 \nabla v(tx+(1-t)y) dt \right) (x-y).
\end{align}
Since the arguments from above show that $\nabla v$ always is a perturbation of an upper triangular matrix, this can only be the case if $x=y$ which hence implies the desired injectivity.

\begin{appendices}

\section{Some Auxiliary Linear Algebra Facts}

We collect some linear algebra facts which are used in the proofs of the main part of our text. 
\begin{lemma}[Representation formula] \label{lem-fform} %
  Let $F \in GL(2)$ be positive-definite, symmetric with $\det F=1$. Then the
  following results hold:
\begin{enumerate}
\item\label{fform1} 
There exist $R\in\SO(2)$ and
    $a,b\in\R^2$ such that
  \begin{align}\label{eq**}
    F=R+a\otimes b.
  \end{align}
\item \label{fform2} There exists $F'=\Id+\nu\otimes e_2$ with
  $\nu=(\nu_1, 0)\in\R^2$ such that
\begin{equation}
\dist(\nabla v, \SO(2)F)=\dist(\nabla v, \SO(2)F').
\end{equation}
\end{enumerate}

\end{lemma}
\begin{proof}
\textit{(i):} Since decomposition \eqref{eq**} does not change under the transformation
  $Q^tFQ$ with $Q\in\SO(2)$ and $\det F=1$, we can assume that
  \begin{align}
      F=\begin{pmatrix}
    \lam & 0\\
    0 & \lam^{-1}
  \end{pmatrix} %
        \qquad \text{and} \qquad %
        R = \begin{pmatrix}
          \cos\varphi & -\sin\varphi\\
          \sin\varphi & \cos\varphi
        \end{pmatrix}.
  \end{align}
  Since $F$ is positive-definite, we have $\lam>0$. In
  view of $0=\det(a\otimes b)=\det(R-F)$, a short calculation then yields
   $\cos{\varphi}=\frac{2}{\lambda+\lambda^{-1}}\le1$. It has a
  solution if and only if
  $\lambda>0$. 
  It proves the claimed decomposition \ref{fform1}.

\medskip

\textit{(ii):} 
By using \ref{fform1}, we have 
\begin{equation}\label{eq formF}
F=R+a\otimes b \text{ for some } a,b\in\R^2 \text{  and } R\in\SO(2).
\end{equation}
We multiply equation \eqref{eq formF} by $R^{-1}$
\begin{equation}
R^{-1}F=\Id+R^{-1}a\otimes b=:\Id+c\otimes b.
\end{equation}
Since $\det R^{-1}F=\det F=1$, we have 
\begin{equation}
1=\det(\Id+c\otimes b)=1+c_1b_1+c_2b_2.
\end{equation}
Therefore, we have $c\perp b$. So there exist a rotation $S\in \SO(2)$  such that
\begin{equation}
SR^{-1}FS^{-1}=\Id+Sc\otimes b S^{-1}=\Id+\nu\otimes e_2.
\end{equation}
It completes the proof of \ref{fform2}.
\end{proof}

For any $F\in{\rm GL}^+(2)$ by polar decomposition there is $R\in\SO(2)$ and
$U=U^t\in\R^{2\times 2}$ positive-definite with $F=RU$. We give two formulas
related to the distance to $\SO(2)$:
  \begin{lemma}[Identities for distance to $\SO(2)$] \label{lem-Fsym} \text{} %
  \begin{enumerate}
  \item\label{Fsym-lem1} For $R\in\SO(2)$ and $U=U^t\in\R^{2\times 2}$ positive
    definite we have
  \begin{align} %
    {\rm dist}(RU, \SO(2)) \ = \NNN{U-{\rm Id}}.
  \end{align}
\item\label{Fsym-lem2} Let $U \in GL(2)$ with
  $\max \{ \NNN{U}{} , \NNN{U^{-1}}{} \} \leq m$ for some $m\geq 1$. Assume
  $A\in \R^{2\times 2}$ is symmetric and positive-definite, then there exists a
  constant $C=C(A,m)>0$ such that
\begin{align}  %
  \dist(U^{-1}, \SO(2) A^{-1}) 
   \ \leq \  C  \dist( U, \SO(2) A). \label{eq:inverse-2}
\end{align}
\end{enumerate}
\end{lemma}
\begin{proof}
  \textit{(i):} This follows e.g. from \cite{Martins-Podio-1979} which states
  that $\|RU-Q\|\ge \|RU-R\|$ for all $Q\in\SO(2)$. 

\medskip

\textit{(ii):} Without loss of generality, we can assume $U\in GL^{+}(2)$,
otherwise all distances are of order $1$ up to a constant depending on $m$ and
$A$. Since $A\in \R^{2\times 2}$ is symmetric, positive-definite and by
\ref{Fsym-lem1}, there exists $S\in\SO(2)$ such that
\begin{equation}
\dist( U, \SO(2)A)=\|\overline{U}-SA\|,
\end{equation}
where $\overline{U}:=\sqrt{U^tU}$. Moreover, we have $\tr(S^tGS)=\tr G$, $\|A\|=\|A^t\|$ and
\begin{equation}\label{m-bound}
\|B\overline{U}^{-1}\| \ \leq \ m\|B\|.
\end{equation}
Using the above last expressions and $\|SA\|=\|A\|$, we hence obtain
\begin{equation}
\begin{split}
\dist( U^{-1}, \SO(2)A^{-1})& \ \leq \ \|\overline{U}^{-1}-SA^{-1}\|=\|\overline{U}^{-1}(A-\overline{U}S)A^{-1}\|\\& \ \stackrel{\eqref{m-bound}}{\leq}  \ \frac{m}{\min\lam_j(A)}\|A-\overline{U}S\|=\frac{m}{\min\lam_j(A)}\|AS^{-1}-\overline{U}\|\\&=\frac{m}{\min\lam_j(A)}\|SA-\overline{U}\|=C\dist( U, \SO(2)A)
\end{split}
\end{equation}
for some constant $C=C(A,m)>0$. Here we used
$(AS^{-1}-\overline{U})^t=SA-\overline{U}$.
\end{proof}

\end{appendices}

\section*{Acknowledgements}

This work received funding from the Heidelberg STRUCTURES Excellence Cluster which is funded by the Deutsche Forschungsgemeinschaft (DFG, German Research Foundation) under Germany's Excellence Strategy EXC 2181/1 - 390900948. M.K. acknowledges funding by the GA\v{C}R project  21-06569K.

\end{document}